\documentclass[preprint,12pt]{elsarticle}

\usepackage{graphicx}
\usepackage{amssymb}
\usepackage[noend]{algorithmic}
\usepackage{algorithm}

\newtheorem{theorem}{Theorem}

\newtheorem{corollary}[theorem]{Corollary}

\newtheorem{definition}[theorem]{Definition}

\newtheorem{lemma}[theorem]{Lemma}

\newenvironment{proof}[1][Proof]{\noindent\textbf{#1.} }{\ \rule{0.5em}{0.5em}}
\def\lab(#1)#2{\put(#1){\makebox(0,0)[c]{#2}}}

\journal{arXiv}

\begin{document}

\begin{frontmatter}

\title{The $2$-connected bottleneck Steiner network problem is NP-hard in any $\ell_p$ plane}

\author[]{M.~Brazil}

\author[]{C.J.~Ras\corref{cor2}}

\author[]{D.A.~Thomas}

\author[]{G.~Xu}


\begin{abstract}
Bottleneck Steiner networks model energy consumption in wireless ad-hoc networks. The task is to design a network spanning a given set of terminals and at most $k$ Steiner points such that the length of the longest edge is minimised. The problem has been extensively studied for the case where an optimal solution is a tree in the Euclidean plane. However, in order to model a wider range of applications, including fault-tolerant networks, it is necessary to consider multi-connectivity constraints for networks embedded in more general metrics. We show that the $2$-connected bottleneck Steiner network problem is NP-hard in any planar $p$-norm and, in fact, if P$\,\neq\,$NP then an optimal solution cannot be approximated to within a ratio of ${2}^\frac{1}{p}-\epsilon$ in polynomial time for any $\epsilon >0$ and $1\leq p< \infty$.

\end{abstract}

\end{frontmatter}

\section{Introduction}
A wireless sensor network (WSN) consists of autonomous and spatially distributed sensing devices that are deployed in diverse environments to collect information and monitor physical conditions, before forwarding the data via multi-hop paths to a base station for processing. An abundance of applications for WSNs (see, eg., \cite{aram,simon,wark}) has fuelled interest in every aspect of the design, function, and deployment of these networks. 

The lifetime of a sensor network (defined as time until network partition) is dependant on the nodes which consume the most power. In turn, the nodes which consume the most power are the nodes that transmit over the largest distances. The problem of designing networks that minimise the length of the longest edge (the bottleneck) is therefore of fundamental importance.

An appropriate model for the WSN lifetime optimisation problem is the geometric bottleneck Steiner network problem, which has been studied extensively for the case where only $1$-connectivity is required; in other words, the constructed networks are trees \cite{abu,bae1,chang,wan}. The so-called \textit{bottleneck Steiner tree problem} was shown in \cite{wan} to be inapproximable to within ratios of less than $\sqrt{2}$ and $2$ in the Euclidean plane and rectilinear plane, respectively. In \cite{wan} Wang and Du also provide a simple heuristic called the ``beaded spanning tree heuristic", which greedily places degree-$2$ Steiner points on the longest edges of a minimum spanning tree. Wang and Du show that their heuristic is at most a $2$-approximation in both the rectilinear and Euclidean planes. Li et al. \cite{li} provide a $\sqrt{3}$-approximation algorithm for the Euclidean bottleneck Steiner tree problem based on a heuristic for finding minimum spanning trees in $3$-regular hypergraphs.


In reality $1$-connectivity is not enough. Survivability is of paramount importance in wireless ad-hoc networks. The nodes of these networks are generally battery powered, and therefore a higher degree of connectivity is required in order to ensure continued function after node depletion. Only three papers have looked at survivable bottleneck Steiner networks: in \cite{brazil2} and \cite{brazil3}, Brazil et al. show that the $2$-connected bottleneck Steiner network problem can be solved in polynomial time when $k$, the number of Steiner points, is constant. In \cite{ras}, Ras uses techniques based on generalised Voronoi diagrams to provide an exact algorithm for the bottleneck Steiner network problem under a very general definition of multi-connectivity. The computational complexity of the $2$-connected bottleneck Steiner network problem when $k$ is part of the input has been an open question until now.

In this paper we demonstrate that the $2$-connected bottleneck Steiner network problem is NP-hard and cannot (unless P=NP) be efficiently approximated to within a ratio of less than $2^\frac{1}{p}$  in planar $\ell_p$ norms (also called $p$-\textit{norms}) when $1\leq p<\infty$. For $p=\infty$, this implies an inapproximability ratio of $2$, since the $\ell_\infty$-plane is simply a $45^\circ$ rotation of the $\ell_1$-plane.

\section{Preliminaries}

We study a formal model of the problem, defined as follows. Given a set $X$ of $n$ points in the plane (called \textit{terminals}) and a positive integer $k$, the \textit{$2$-connected bottleneck $k$-Steiner network problem} asks for network $N$ of minimum \textit{bottleneck} (longest edge) length such that $N$ spans $X$ and at most $k$ additional points, and for every pair of nodes $u,v$ in $X$ the number of internally node-disjoint paths connecting $u$ and $v$ in $N$ is at least $2$. Length is measured in the $p$-{norm}, which, for any vector $e=(x,y)\in\mathbb{R}^2$, we denote as $\|e\|_p:=(\vert x\vert^p+\vert y\vert^p)^{\frac{1}{p}}$ Network $N$ is called a \textit{minimum $2$-connected bottleneck $k$-Steiner network}.



We first state a number of definitions and preliminary results.

\begin{definition}The unit circle of the $p$-norm is the set of points $\{e\in\mathbb{R}^2\mathrm{\ s.t.\ }\|e\|_p=1\}$.
\end{definition}

It is easy to show that the unit circle of the $p$-norm 
has $90^\circ$ rotational symmetry.




\begin{definition}A full Steiner tree of a Steiner network $N$ is subtree $T$ of $N$ such that every terminal is of degree $1$ in $T$ and every Steiner point is of the same degree in $T$ as it is in $N$.
\end{definition}

Note that an edge connecting two terminals is a full Steiner tree according to the above definition.

\begin{lemma}[\cite{luebke}]\label{lemLu}There exists a minimum $2$-connected bottleneck $k$-Steiner network $N$ on $X$ such that the edge-set of $N$ can be partitioned into full Steiner trees.
\end{lemma}


\begin{figure}[htb]
\begin{center}
\includegraphics[width=6cm]{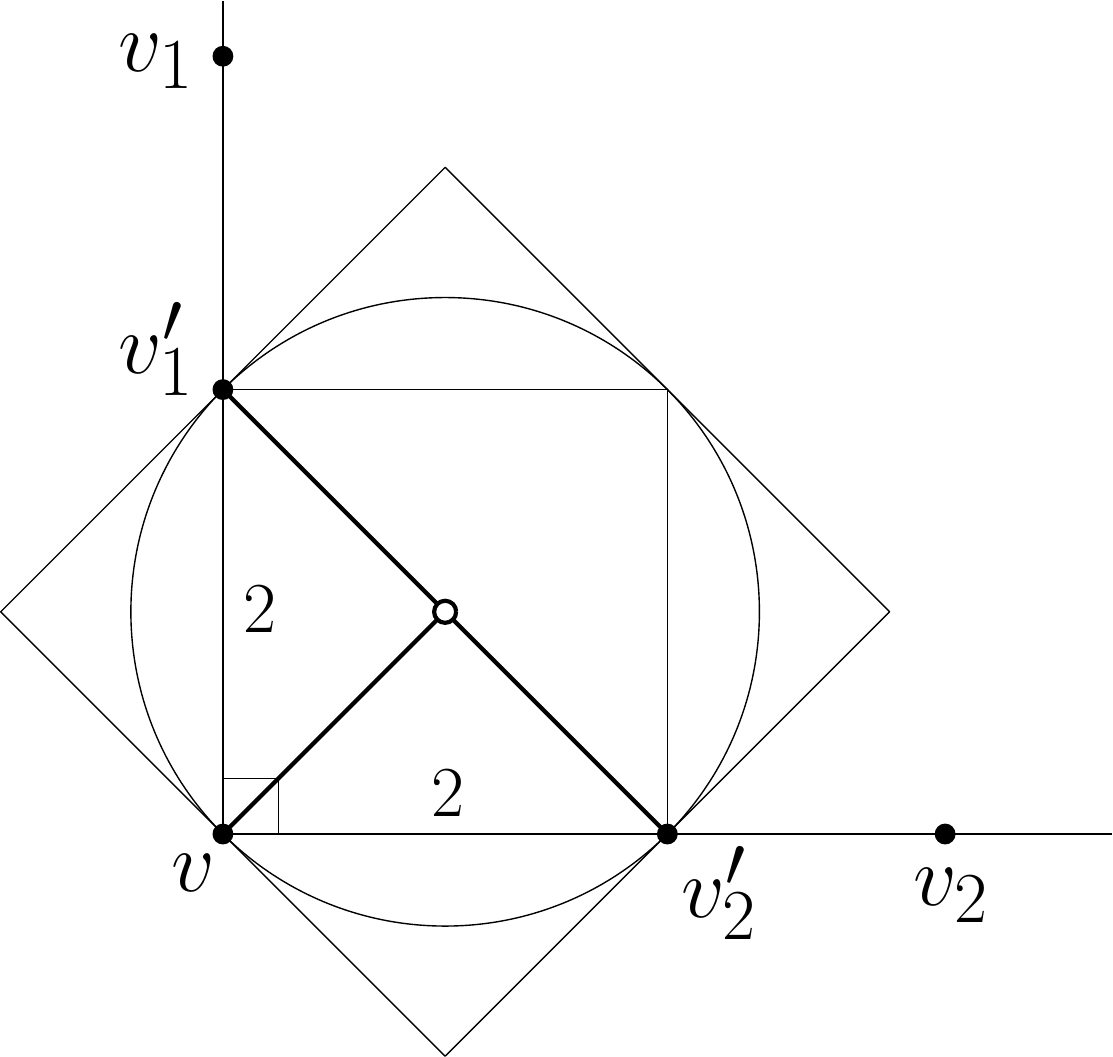}
\caption{Illustration for the proof of Lemma \ref{mainLem}}
\label{figPRatio}
\end{center}
\end{figure}

The following lemma will be used in our main proof:

\begin{lemma}\label{mainLem}Let $v,v_1,v_2$ be three points in the plane such that $v_1{v}$ is parallel to the $y$-axis, $v_2v$ is parallel to the $x$-axis, $\|v-v_1\|_p\geq 2$ and $\|v-v_2\|_p\geq 2$. Let $T$ be a full Steiner tree of minimum bottleneck length on $v,v_1,v_2$ such that $T$ contains a single Steiner point. Then the length of the bottleneck edge in $T$ is at least $2^{\frac{1}{p}}$.
\end{lemma}
\begin{proof}
Observe first that the bottleneck edge of any optimal full Steiner tree on three terminals and a single Steiner point has the same length as the radius of a smallest enclosing circle (w.r.t. the $p$-norm) of the three terminals. Let $B$ be a smallest circle such that $B$ encloses the points $v,v_1,v_2$. Since $B$ is convex, the line segments $vv_1$ and $vv_2$ lie in $B$. Therefore the points $v_1',v_2'$ lie in $B$, where, for $i\in\{1,2\}$, $v_i'$ is at a distance of exactly $2$ from $v$ and lies on segment $vv_i$. Hence the length of the bottleneck in $T'$ is no more than the length of the bottleneck in $T$, where $T'$ is an optimal full Steiner tree with a single Steiner point connecting $v,v_1',v_2'$.

The lemma clearly holds for $p=\infty$, since in this case $\|v-v_1'\|=\|v-v_2'\|=\|v_1'-v_2'\|=2$ and a smallest enclosing circle for $v,v_1',v_2'$ exists with radius $1$ and centre at the midpoint of $v_1'v_2'$. Therefore assume that $p<\infty$. Observe then that a smallest circle enclosing $v$ and $v_1'$ has its centre at the midpoint of segment $vv_1'$, and therefore does not include $v_2'$. We claim that a smallest circle enclosing $v_1'$ and $v_2'$ (note, there exists at least one such circle which is centred at the midpoint of $v_1'v_2'$) also includes the point $v$. But this follows from the $90^\circ$ rotational symmetry of the unit circle of the $p$-norm; see Figure \ref{figPRatio}. Therefore the radius of a smallest enclosing circle of $v,v_1',v_2'$ is $\frac{1}{2}(2^p+2^p)^\frac{1}{p}=2^\frac{1}{p}$.
\end{proof}

\section{Approximability analysis}
We show that it is NP-hard to approximate the $2$-connected $k$-bottleneck Steiner network problem to within a ratio smaller than ${2}^\frac{1}{p}$ when $1\leq p< \infty$. The reduction is from the following NP-complete problem \cite{ans}.

{\textsc{Hamiltonian cycle in $2$-connected, cubic, bipartite planar graphs}\\
	\indent {\textbf{Instance:}} A $2$-connected, cubic, bipartite planar graph $G$.\\
  \indent {\textbf{Question:}} Does $G$ contain a Hamiltonian cycle?

\begin{theorem}
\label{th2}
  It is \emph{NP}-hard to approximate the $2$-connected $k$-bottleneck Steiner network problem to within a ratio smaller than ${2}^\frac{1}{p}$ when $1\leq p< \infty$.
\end{theorem}

\begin{proof}
Let $G=(V, E)$ be a 2-connected, cubic, bipartite planar graph, where $V=U\cup W$ is the bipartition, and suppose that the $2$-connected bottleneck $k$-Steiner network problem has a $({2}^\frac{1}{p}-\epsilon)$-approximation algorithm $\mathcal{A}$, where $\epsilon>0$. Let $n=|V|/2$. We construct a set $X$ of terminal points in the plane such $G$ has a Hamiltonian cycle if and only if $\mathcal{A}$ produces a $2$-connected network $N(\mathcal{A})$ spanning $X$ and at most $k:=2n$ Steiner points such that the longest edge in $N(\mathcal{A})$ is of length at most ${2}^\frac{1}{p}-\epsilon$.

Since $G$ is bipartite and cubic, we have that each part of the bipartition
$V=U\cup W$ has $n$ vertices and $|E|=3n$.
Let $U:=\{u_1, u_2, \ldots, u_{n}\}$ and  $W:=\{w_1, w_2, \ldots, w_{n}\}$. For each $u\in U$, let $E(u)$ be the set of three edges incident to $u$. Then  $\{E(u) : u\in U\}$ forms a partition of $E$ into triples.

The first step is to orthogonally embed $G$ in the plane. We do this by mapping each vertex of $V$ to a distinct integer point in the plane such that the minimum horizontal or vertical distance between any two parallel line segments (parts of edges of $G$) is at least $\Delta := 4n+2$. Note that such a representation of $G$ takes a polynomial amount of time to create \cite{tt} and the coordinates of $G$ are bounded by a polynomial in $n$.

\begin{figure}[htb]
\begin{center}
\includegraphics[width=9cm]{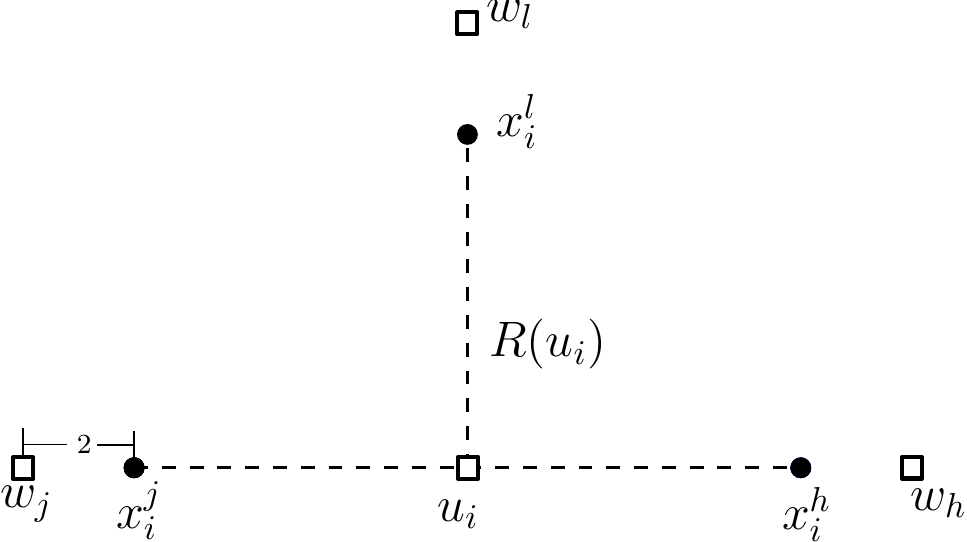}
\caption{$R(u_{i})$, represented by dashed lines, for the edge set $E(u_i)$. The black-filled circles are terminals of $N$ and are called tips of $R(u_{i})$.}
\label{fig:Rui}
\end{center}
\end{figure}

The terminal set $X$ is constructed as follows. For each $w\in W$, let $p_{w}$ be the corresponding grid point in the embedding, and call $p_{w}$ a {\em $W$-terminal}. For each $u_{i}\in U$, let $w_{j}$,  $w_{h}$ and $w_{l}$ be the three neighbours of $u_{i}$. Note that there is a grid path in $G$ connecting $u_i$ and each of  $w_{j}$,  $w_{h}$ and $w_{l}$. Place a terminal $x_{i}^s$ on the grid path between $u_i$ and $w_{s}$, for each $s\in\{j,h, l\}$, such that the distance between $x_{i}^s$ and $w_{s}$ is exactly $2$. Let $R(u_{i})$ be the union of the parts of the grid paths connecting each pair of points from $x_{i}^j$, $x_{i}^h$, $x_{i}^l$. See Figure \ref{fig:Rui} for an illustration. We call each $x_{i}^s$ a {\em tip} of $R(u_{i})$, where $s\in\{j,h, l\}$. For distinct $i, j$, $R(u_{i})$ and $R(u_{j})$ are said to be {\em adjacent} if $u_i$ and $u_j$ share a common neighbour in $G$.

\begin{figure}[htb]
\begin{center}
\includegraphics[width=9cm]{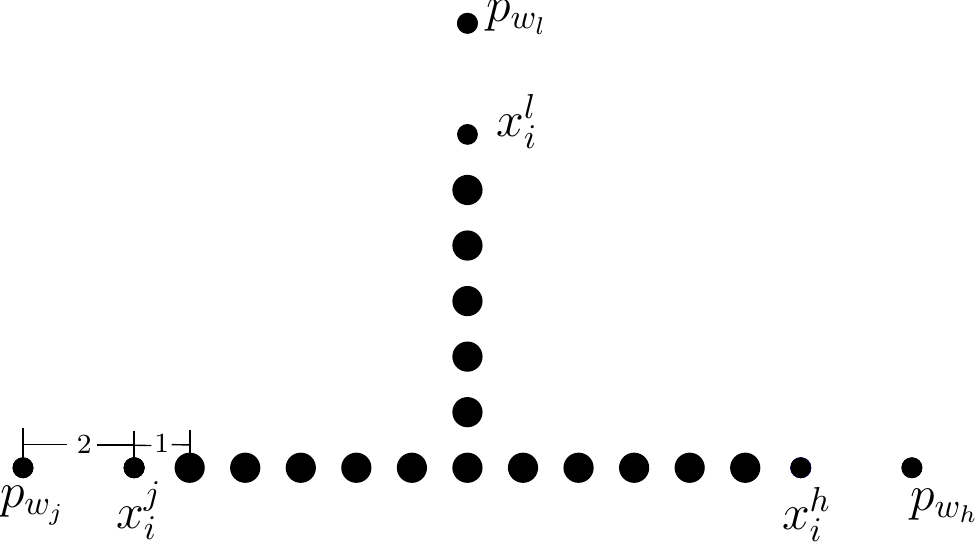}
\caption{Many terminals are placed on $R(u_i)$. Consecutive pairs of coincident terminals are at a distance of $1$ from each other. Large black-filled circles represent pairs of coincident terminals.}
\label{fig:tshape}
\end{center}
\end{figure}

Next, place two terminals on $R(u_{i})$ at a distance of exactly $1$ from each $x_{i}^s$, $s\in\{j,h,l\}$ (note that the locations of these two terminals coincide). Also, place many pairs of coincident terminals on $R(u_{i})$ so that the distance between any two pairs of consecutive coincident terminals is  $1$ (see Figure \ref{fig:tshape}). Let the set of all coincident pairs of terminals together with the three tips on $R(u_i)$ be denoted by $P(u_{i})$. Finally, let $X:=(\bigcup_{w\in W}p_{w})\cup(\bigcup_{u\in U} P(u))$.

We now prove that $G$ has a Hamiltonian cycle if and only if, using $k=2n$ Steiner points, algorithm $\mathcal{A}$ produces a $2$-connected network $N(\mathcal{A})$ on $X$ of bottleneck length at most ${2}^\frac{1}{p}-\epsilon$.


\begin{figure}[htb]
\begin{center}
\includegraphics[width=8cm]{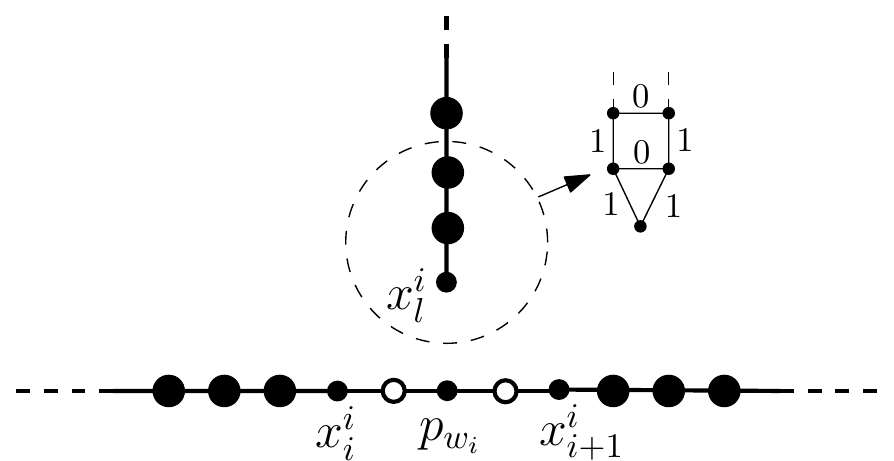}
\caption{Steiner points (white circles) placed between tip-terminals $x_{i}^i$ and  $x_{i+1}^i$. }
\label{fig:spoint}
\end{center}
\end{figure}

Suppose now that $G$ has a Hamiltonian cycle $C= u_1, w_1, u_2, w_2, \ldots, u_{n}, w_{n}, u_1$. Note that for each $1\le i \le n$,  $x_{i}^i$ and  $x_{i+1}^i$ are two tip terminals on $R(u_{i})$ and $R(u_{i+1})$ respectively, and each is at a distance of $2$ to $p_{w_i}$, where the label $n+1$ is read as $1$. Place one Steiner point at the midpoint of $x_{i}^i$ and $p_{w_{i}}$, and one Steiner point at the midpoint of $p_{w_{i}}$ and $x_{i+1}^i$. Now add edges between all terminals at distance of at most $1$ from each other (see Figure \ref{fig:spoint}). Note that the degree of every terminal is at least $2$.

Denote the resultant graph by $N^*$. Clearly the subgraph of $N^*$ induced by the terminals of $P(u_i)$ is $2$-connected (see the magnified region in Figure \ref{fig:spoint}). Also, since $C$ is a Hamiltonian cycle, if we contract every set $P(u_i)$ to a single node we obtain a cycle passing through every node. Therefore, every pair of terminals in $N^*$ lies on a common cycle. Hence, $N^*$ is $2$-connected. Finally, note that the length of a bottleneck edge in $N^*$ is at most $1$, and the total number of Steiner points added is $2\times n=k$. Therefore, since the approximation ratio of algorithm $\mathcal{A}$ is ${2}^\frac{1}{p}-\epsilon$, the length of a bottleneck edge in a network $N(\mathcal{A})$ constructed by algorithm $\mathcal{A}$ is at most $2^\frac{1}{p}-\epsilon$.

Conversely, suppose that algorithm $\mathcal{A}$ constructs a $2$-connected network $N(\mathcal{A})$ on $X$, using $2n$ Steiner points, such that the bottleneck in $N(\mathcal{A})$ is of length at most $2^\frac{1}{p}-\epsilon$. Note first that the distance between a terminal in $P(u_{i})$ and a terminal in $P(u_{j})$ is at least $2^{\frac{1}{p}+1}$ if $R(u_{i})$ and $R(u_{j})$ are adjacent (this is the smallest distance between two tips; for instance, the distance between $x_i^i$ and $x_l^i$ in Figure \ref{fig:spoint}). Also, the distance between a terminal in $P(u_{i})$ and a terminal in $P(u_{j})$ is at least $\Delta=4n+2$ if $R(u_{i})$ and $R(u_{j})$ are not adjacent; thus, since $(2^\frac{1}{p}-\epsilon)(2n+1)<4n+2$, they cannot be connected using edges of length less than $2^\frac{1}{p}$ using $k=2n$ Steiner points. Therefore no full Steiner tree of $N(\mathcal{A})$ joins terminals of distinct non-adjacent $P(u_{i})$ and $P(u_{j})$. That is, each pair of terminals belonging to the same full Steiner tree in $N(\mathcal{A})$ lie in either the same $P(u_{i})$ or in adjacent $P(u_{i})$ and $P(u_{j})$. Without loss of generality, we assume that all edges of length at most $1$ connecting terminals in the same $P(u_{i})$ are in $N(\mathcal{A})$.

Since $N(\mathcal{A})$ is $2$-connected, each $W$-terminal must be incident to at least two edges from distinct full Steiner trees (see Lemma \ref{lemLu}). We claim that, indeed, each $p_{w_j}$ is connected to exactly two tip-terminals of distinct $P(u_{i})$ by disjoint paths, each containing a single Steiner point. First note that $p_{w_j}$ is at distance of at least $2$ to any other terminal, which means whenever $p_{w_j}$ is connected to a terminal by some path then there must be at least one Steiner point lying on the path. Since there are $2n$ Steiner points in total in $N(\mathcal{A})$ and $n$ $W$-terminals, no $W$-terminal is contained in more than two distinct full Steiner trees.

Next we show that $W$-terminals cannot be connected to two terminals of the same $P(u_{i})$. This can be easily seen from the construction of $X$, since the next closest terminal to any $W$-terminal (after a tip) is at a distance of at least $3$.

Finally, we observe also that no $W$ terminal can be connected to two terminals from distinct $P(u_{i})$ using a single Steiner point. The shortest bottleneck for an full Steiner trees of this form is ${2}^\frac{1}{p}$, as proved in Lemma \ref{mainLem}. Therefore each $W$-terminal is contained in exactly two distinct full Steiner trees, both of which are paths connecting to distinct tips and each of which contains exactly one Steiner point.

Now, for the original graph $G$, form an edge set $E'\subset E$ as follows: for each Steiner point connecting two terminals $u_{i}^j$ and $p(w_{j})$, we add the edge $u_{i}w_{j}$ of $G$ to $E'$. Let the graph $C$ be obtained from $N(\mathcal{A})$ by relabelling each vertex $p_{w_j}$ as $w_j$, and contracting each $P(u_{i})$ and its two adjacent Steiner points into a single vertex $u_i$. It is not hard to see that $C$ is a cycle of length $n$ meeting each vertex of $U\cup V$ and $C$ is isomorphic to the graph induced by edges in $E'$. That is, $E'$ gives rise to a Hamiltonian cycle of $G$.
\end{proof}

As mentioned in the introduction, in the case of the $1$-connected bottleneck Steiner problem, there exists a simple $2$-approximation algorithm in the Euclidean and rectilinear norms \cite{wan} which greedily places degree-$2$ Steiner points on the longest edges of a minimum spanning tree interconnecting the given set of terminals. The question arises as to whether an analogous approximation algorithm can be designed for the $2$-connected bottleneck Steiner network problem. There are two obstacles to this potential approach: firstly, although minimum spanning trees can be constructed in polynomial time, the minimum $2$-connected spanning network problem (where Steiner points are \textit{not} allowed and network cost is measured as the sum of all edge lengths) is NP-hard in the Euclidean plane \cite{czu}. We expect the same to be true in other $p$-norms. Furthermore, as stated in the next corollary, even if we restrict the degree of Steiner points to $2$, the $2$-connected bottleneck Steiner network problem remains NP-hard.

\begin{corollary}It is NP-hard to approximate the {$2$-connected bottleneck $k$-Steiner network problem} to within a ratio smaller than $2^\frac{1}{p}$ in polynomial time, even if all Steiner points are constrained to degree $2$.
\end{corollary}
\begin{proof}
Observe that the proof of Theorem \ref{th2} can be used almost verbatim for the degree-$2$ restricted case. We simply omit the case represented by Figure \ref{figPRatio}. The inapproximability ratio follows from the fact that the smallest distance between two non-consecutive terminals of $X$ is the distance between tips of two adjacent $P(u_i)$, which is $2\times 2^\frac{1}{p}$.
\end{proof}


\end{document}